\documentclass[reqno,11pt]{amsart}

\usepackage{amssymb}
\usepackage{amsmath}
\usepackage{amsfonts}
\usepackage{mathtools}
\usepackage{esint}
\usepackage{graphicx}
\usepackage{stackrel}
\usepackage{xcolor}

\usepackage{a4wide}

\usepackage[shortlabels]{enumitem}
\usepackage{comment}
\usepackage{amsthm}
\usepackage[colorlinks=true, citecolor=blue]{hyperref} 
\usepackage[capitalise, nameinlink, noabbrev]{cleveref} 

\newtheorem{thm}{Theorem}[section]
\newtheorem{cor}[thm]{Corollary}
\newtheorem{lem}[thm]{Lemma}
\newtheorem{pro}[thm]{Proposition}

\numberwithin{equation}{section}

\theoremstyle{definition}

\newtheorem*{ack}{Acknowledgements}

\DeclareMathOperator{\sect}{sec}
\DeclareMathOperator{\Ric}{Ric}
\DeclareMathOperator{\diam}{diam}
\DeclareMathOperator{\con}{Con}

\newcommand{\rcd}[2]{\mathsf{RCD}(#1,#2)}
\newcommand{\norm}[1]{\left\Vert#1\right\Vert}
\newcommand{\abs}[1]{\left\vert#1\right\vert}
\newcommand{\inner}[2]{\left\langle \nabla #1, \nabla #2\right\rangle}
\renewcommand{\d}{\mathrm{d}}

\begin{document}


\title[Behavior of harmonic functions at singular points]{The behavior of harmonic functions at singular points of  $\mathsf{RCD}$ spaces.}

\author[G. De Philippis]{Guido De Philippis}
\address{G.D.P.: SISSA, Via Bonomea 265, 34136 Trieste, Italy}
\email{guido.dephilippis@sissa.it}

\author[J.~N\'u\~nez-Zimbr\'on]{Jes\'us~~N\'u\~nez-Zimbr\'on}
\address{J.~N.-Z.: Centro de Ciencias Matem\'aticas UNAM, Antigua Carretera a P\'atzcuaro 8701, 58089 Morelia, Michoac\'an, Mexico}
\email{zimbron@matmor.unam.mx}


\begin{abstract}
In this note we investigate the behavior of harmonic functions at singular points of $\mathsf{RCD}(K,N)$ spaces. In particular we show that their gradient  vanishes at all points where the tangent cone is  isometric to a cone over a metric measure space with non-maximal diameter. The same analysis is performed for functions with Laplacian in $L^{N+\varepsilon}$. As a consequence we show that on smooth manifolds there is no a priori estimate on the  modulus of continuity of the gradient of harmonic functions which depends only on lower bounds of the sectional curvature. In the same way we show that there is no a priori  Calder\'on-Zygmund theory for the Laplacian with bounds depending only on lower bounds of the sectional curvature.

\end{abstract}

\maketitle


\section{Introduction}

\subsection*{Background}

One of the main themes of research in Geometric Analysis consists in establishing \emph{a priori} estimates for solutions of PDE that depend only on the geometric data of the problem (upper/lower bounds on curvature, upper bounds on the diameter, upper bounds on the dimension, etc), see for instance \cite{PLi} for a (non-complete) account of the theory. Among the various results which can be obtained let us mention the following ones:

\begin{itemize}
\item[-]\emph{Lipschitz regularity of harmonic functions}: If \(M\) is an \(N\)-dimensional Riemannian manifold with \(\Ric\ge K\), then every harmonic function is locally Lipschitz. More precisely, there exists a constant \(C=C(K,N)\) such that 
\begin{equation}\label{e:lipharmint}
\sup_{B_{R/2}(x_0)} |\nabla u|^2 \leq C \fint_{B_{R}(x_0)} |\nabla u|^2. 
\end{equation}
for all functions \(u\) such that \(\Delta u=0\) on \(B_{R}(x_0)\).

\item[-] \emph{Calder\'on-Zygmund theory}:  If \(M\) is an \(N\)-dimensional Riemannian manifold with \(\Ric\ge K\), then there exists a constant  \(C=C(K)\) such that  
\begin{equation}\label{e:czint}
\norm{\nabla^2 u}_{L^2}\le C(\norm{\Delta u}_{L^2}+\norm{\nabla u}_{L^2}).
\end{equation}
for all functions \(u\in C_c^\infty(M)\).
\end{itemize}

Besides their intrinsic interest, these inequalities  can be used to establish various properties of Gromov-Hausdorff limits of manifolds uniformly satisfying the above bounds. Indeed, the Gromov pre-compactness theorem ensures that the class of manifolds with Ricci curvature bounded from below by \(K\), dimension bounded above by \(N\) and diameter bounded above by \(D\) is pre-compact with respect to Gromov-Hausdorff convergence. The above estimates, when suitably interpreted, can be passed to the limit and thus they hold true on the so-called \emph{Ricci limit} spaces studied by Cheeger and Colding, \cite{Cheeger-Colding96, Cheeger-Colding97I, Cheeger-Colding97II, Cheeger-Colding97III, Colding96, Colding96b, Colding97}.

A different point of view consists in establishing a \emph{synthetic} notion of lower curvature bounds which is valid for arbitrary metric spaces and that is stable under Gromov-Hausdorff convergence and to study directly which implications these bounds have on the geometry of the spaces. In particular, whether inequalities in the spirit of \eqref{e:lipharmint} and \eqref{e:czint} are valid (again suitably interpreted). In the case of lower bounds on sectional curvature, this idea led to the theory of Alexandrov spaces, \cite{BuragoBuragoIvanov, BuragoGromovPerelman}.

At the beginning of the 2000's, Lott, Sturm and Villani introduced the notion of   \(\mathsf{CD}(K,N)\) spaces which provided a synthetic notion of Ricci curvature bounded below by \(K\) and dimension bounded above by \(N\). This class of spaces is closed with respect to (measure) Gromov-Hausdorff convergence and, in particular, contains all Ricci limit spaces and all Alexandrov spaces, \cite{LotVil, Stu1, Stu2}. However, these spaces include also  Finsler manifolds, \cite{Ohta}. In order to obtain the analogues of \eqref{e:lipharmint} and \eqref{e:czint} one thus has to:

\begin{enumerate}[label=(\alph*)]
\item Find a suitable restriction on the class of spaces which ensures the validity of a Bochner inequality.
\item Develop a suitable calculus to give a meaning to the objects appearing in the above inequality. 
\end{enumerate}

In the seminal works   \cite{AmbrosioGigliSavare11-2, GigDiffStruc} as a restriction in item  (a) above it has been proposed to impose the \emph{infinitesimal Hilbertianity} of the space, that is, that the Sobolev space \(W^{1,2}\) is Hilbert, leading to the notion of \(\rcd{K}{N}\) spaces. When infinitesimal Hilbertianity is coupled with a lower bound on the Ricci curvature, it is  stable under measure Gromov-Hausdorff convergence  and  rules out Finsler geometries. In particular it has been shown to imply the validity of a Bochner type inequality, \cite{EKS}. Parallel to this a suitable first and second order  calculus on metric spaces has been developed in great detail, \cite{Gigli14, GigDiffStruc}. In particular    \eqref{e:lipharmint} and \eqref{e:czint} have been shown to be true on \(\rcd{K}{N}\) spaces, \cite{GigDiffStruc,Kell}. We refer the reader to the recent survey \cite{AmbrosioICM} for  a detailed account of the main results in the field.

\subsection*{Main results} In this note we aim to understand  if it is possible to improve upon  \eqref{e:lipharmint}  and \eqref{e:czint}. In particular we aim to answer to the following questions:

\begin{enumerate}[label=(Q.\arabic*)]
\item \emph{Is it possible to find a modulus of continuity for the gradient of harmonic functions that depends only on a lower bound on the Ricci curvature and an upper bound of the dimension?}\label{q:Q1}
\item \emph{Is it possible to  obtain the validity of \eqref{e:czint} for some \(p\ne 2\)?}\label{q:Q2}
\end{enumerate}

We actually show that the answer to both questions is indeed negative, also if one assumes a lower bound on sectional curvature. Our  proof  relies on the aforementioned theory of \(\rcd{K}{N}\) spaces and our main result studies the behavior of the gradient of harmonic functions at singular points of these spaces, a result which is  interesting on its own. In order to state our main result  let us recall that for any given metric measure space $(X,d,m)$ satisfying the $\rcd{K}{N}$ condition for $K\in \mathbb{R}$, $N\in [1,\infty)$ the following holds true: For any point $x\in X$  the  \emph{Bishop-Gromov density} 
\begin{equation}\label{e:bg}
\vartheta(x)=\lim_{r\to0}\frac{m(B_R(x))}{R^N}\in (0,+\infty]
\end{equation}
exists. Moreover if \(\vartheta(x)<+\infty\) every tangent cone at $x$ is an $N$-metric measure cone, \cite{DePGigNonCol}. If its base space  has diameter $<\pi$ then the cone is singular. We will  call such a cone,  \textit{sharp}.

\begin{thm}
\label{thm-mwug-vanishes}
 Let $(X,d,m)$ be an $\mathsf{RCD}(K,N)$ space and $u\in D(\Delta)$ with $\Delta u\in L^{p}(X,m)$ for some $p>N$. Then 
  \[
 \lim_{R \to 0} \fint_{B_R(x)}|\nabla u|^2=0
 \]
at any point $x\in X$ with finite Bishop-Gromov density such that every tangent cone is sharp.  
\end{thm}

We remark that the assumption on the finiteness of the Bishop-Gromov density is satisfied at all points for non-collapsed $\mathsf{RCD}$ spaces, see \cite{DePGigNonCol}. In particular the above result holds true  for finite-dimensional Alexandrov spaces when considered as metric measure spaces equipped with the Hausdorff measure. Points with sharp tangent cones were called \textit{essentially singular points} in \cite{MitYam}. Let us further remark that once a half-Harnack inequality of the form
\[
|\nabla u|^2(x)\lesssim  \fint_{B_R(x)}|\nabla u|^2+R^{2(1-\frac{N}{p})}\Bigl(\int_{B_R(x)} |\Delta u|^p\Bigr)^{\frac{2}{p}}\qquad p>N,\quad R\ll1 
\]
 is available, it follows from \cref{thm-mwug-vanishes} that $\abs{\nabla u}(x)=0$ at sharp points under the same assumptions. This inequality can be obtained in the context of $\mathsf{RCD}(K,N)$ spaces by the same techniques used in the proof of  \cite[Theorem 3.1]{JiangKoskelaYang}, since we do not need it in the sequel, we do not insist on this point.

Being  easy to build metric  measure spaces with a dense set of sharp points (see for instance Example (2) in \cite{OtsuShioya} for a $2$-dimensional example which can be easily generalized to any dimension),  \cref{thm-mwug-vanishes} implies that for a non constant  harmonic function \(u\), \(|\nabla u|\) cannot be continuous.
Furthermore, it is  easy  to construct a sequence of smooth manifolds with non-negative sectional curvature converging to  an Alexandrov space with a dense set of sharp  points. This implies the  following two corollaries which answer negatively \ref{q:Q1} and \ref{q:Q2}, at least for \(p\) large. Note that  \cref{cor:noharm} actually excludes every possible modulus of continuity also on the norm of the gradient. 

\begin{cor}\label{cor:noharm}
Let \(\omega\) be a modulus of continuity\footnote{Recall that a modulus of continuity is a continuous and bounded function \(\omega: (0,+\infty)\to (0,1]\) such that \(\omega(0+)=0\).} and let \(N\in \mathbb N\) and \(D\ge 2\). Then 
there exist  sequences of \(N\)-dimensional  Riemannian manifolds \((X_k, g_k)\)  with \(\diam (X_k)\le D \), \(\sect(X_k)\ge 0\), points \(p_k\in X_k\)   and harmonic functions  \(f_k:B_1^{X_k}(p_k)\to \mathbb R\)  such that 
\[
\|f_{k}\|_{W^{1,2}(B_1^{X_k}(p_k))}=1
\]
and 
\[
\lim_{k\to +\infty}\sup_{p,q\in B_{1/2}^{X_k}(p_k)}\frac{\bigl| | \nabla f_k|(p)|- |\nabla f_k|(q)|\bigr|}{\omega(d_{g_k}(p,q))}=+\infty.
\]
\end{cor}

\begin{cor}\label{cor:nocz}
Let  \(N\in \mathbb N\), \(D\ge 2\) and  \(p>N\). Then  there exist  sequences of \(N\)-dimensional  Riemannian manifolds \((X_k, g_k)\)  with \(\diam (X_k)\le D \), \(\sect(X_k)\ge 0\) and of smooth functions \(f\in C^\infty(X_k)\) such that 
\[
\|\Delta f_{k}\|_{L^p}+\|\nabla f_k\|_{L^p}=1
\]
and such that
\[
\lim_{k\to +\infty}\|\nabla^2 f_{k}\|_{L^p}=+\infty.
\]
\end{cor}
In particular, for \(p>N\),  \eqref{e:czint} cannot be true with a constant which depends only on a lower bound on the sectional curvature and on an upper bound on the volume.

Let us note that the validity of a Calder\'on-Zygmund theory with constants depending only on the geometric data of the manifold has been investigated in \cite{GunPig}, where  positive results were obtained assuming a bound on the whole Ricci tensor and on the injectivity radius of the manifold. Examples of non-compact manifolds where a Calder\'on-Zygmund theory fails can be found in \cite{GunPig} and in \cite{Li}. However, to the best of our knowledge, the results in  \cref{cor:nocz} are the first to show that it is impossible to have a Calder\'on-Zygmund theory with constants depending only on a lower bound on the sectional curvature.

We conclude this introduction by briefly explaining the strategy of the proof of  \cref{thm-mwug-vanishes} which is based on a simple perturbation argument. Indeed, by the explicit form of harmonic functions on a cone \(C=\con_{O}^N(X)\) one easily deduces that  if the cone is singular and \(v\) is harmonic, then 
\[
\fint_{B_{1/2}} |\nabla v|^2\le(1-\delta)  \fint_{B_{1}} |\nabla v|^2,
\]
where \(\delta\) is a positive constant that depends only on the gap \((\pi-\diam X)\). Since the sequence of rescaled spaces \((X,d/r,m/r^N)\) converges to \(C\), a simple compactness argument implies the validity of an estimate of the form
\[
\fint_{B_{\bar {r}/2}} |\nabla u|^2\le(1-\delta)  \fint_{B_{\bar{r}}} |\nabla u|^2 +C  \fint_{B_{\bar{r}}}|\Delta  u|^{2}
\] 
for a certain small and fixed radius \(\bar r\). Iterating, one easily deduces that, if \(\Delta u\in L^p\) for \(p>N\), then  \(\fint_{B_r} |\nabla u|^2\to 0\) as \(r \to 0\).

\begin{ack}
We thank Nicola Gigli for encouraging us to write this note and for several useful discussions. G.D.P. is supported by an INDAM grant ``Geometric Variational Problems''. J.N.-Z. was supported by a MIUR SIR-grant ``Nonsmooth Differential Geometry'' (RBSI147UG4) and a DGAPA-UNAM Postdoctoral Fellowship. J.N.-Z. also acknowledges support from CONACyT project CB2016-283988-F.
\end{ack}


\section{Proof of  \cref{thm-mwug-vanishes}}

We assume the reader to be familiar with the theory of \(\rcd{K}{N}\) spaces; we refer to \cite{GigLecNotes} for the general theory and the main notations. Here we simply recall  that if \((X,d_X, m_X)\) is a metric measure space, then \(\mathrm{Con}_0^N(X):=(C, d_c, m_c)\) is the metric measure space given by the warping product  \(X\times_{r^{N-1}} \mathbb R\), see \cite{Ket}. In particular, the measure \(m_c\) is characterized by the equality
\begin{equation}\label{e:measurewarping}
\int f(r, x)\, \d m_{c}=\int_0^{+\infty} r^{N-1}\int_{X} f(r,x)\, \d m_X(x) \d r\qquad \text{for all \(f\) Borel}.
\end{equation}
Moreover, if \(f\in W^{1,2}( \mathrm{Con}_0^N)\), then
\begin{equation}\label{e:gradientwarping}
|D f|^2_{c}(r,x)=r^{-2}|D f^{(r)}|_X^2(x)+|D f^{(x)}|^2_ {\mathbb R}(r),
\end{equation}
where \(f^{(r)}:=f(r,\cdot): X \to \mathbb R \) and \(f^{(x)}=f(\cdot,x) :\mathbb R\to \mathbb R\). Here, $|Df|_c$, $|Df^{(r)}|_X$ and $|Df^{(x)}|_{\mathbb{R}}$ denote the minimal weak upper gradients of the corresponding functions on $\mathrm{Con}_0^N(X)$, $X$ and $\mathbb{R}$ respectively, see \cite{GigHan}. Given a metric space \((X,d)\) we denote by \(B^X_R(x)\) (or if the space is clear from the context by \(B_R(x)\)) the open ball centered at \(x\).

The following decomposition for harmonic functions defined on a cone easily follows by separation of variables and spectral theory, see  \cite[Theorem 3.1]{Huang}. 

\begin{pro}
\label{pro-charact-harmonic-fns}
Let $N\geq 2$, $(X,d,m)$ an $\rcd{N-2}{N-1}$ space and $C:=\mathrm{Con}_0^N(X)$ be the metric measure cone constructed over \(X\). If $u:B_{1}(O) \subset C\to\mathbb{R}$ is  a \(W^{1,2}\) harmonic function, then there exist $\{a_i\}_{i}\in \mathbb{R}$ such that  
\begin{equation}\label{e:dec}
u(r,x)=\sum_{i=0}^{\infty} a_i r^{\alpha_i} \varphi_i(x),
\end{equation}
where $O\in C$ is the vertex of $C$, \(\varphi_i: X \to \mathbb R\) is an eigenfunction of \(-\Delta_X\) (the Laplacian on \(X\)) with eigenvalue \(\lambda_i\)  and 
\begin{equation}\label{e:alpha}
\alpha_i=\frac{-(N-2)+\sqrt{(N-2)^2+4\lambda_i}}{2}.
\end{equation}
Here the convergence in \eqref{e:dec} is intended in \(W^{1,2}(B_1(O))\) and we are ordering the eigenvalues as:
\[
0=\lambda_0\le\lambda_1\le\dots\le \lambda_i\le \dots
\]
\end{pro}

Recall that from the non smooth version of the Lichernowitz-Obata Theorem, \cite{KetObata}, we get that  if \(X\) is $\rcd{N-2}{N-1}$, then \(\lambda_1\ge (N-1)\), which in turn implies that \(\alpha_i\ge \alpha_1\ge 1\). A simple compactness argument shows that this bound can be improved if we assume that \(\diam(X)<\pi\).

\begin{lem} 
\label{lem-quant-eigenvalue-bound}
Let $(X,d,m)$ be an $\rcd{N-2}{N-1}$ space. For every $0<\varepsilon<\pi$ there exists $\delta_1=\delta_1(N,\varepsilon)>0$ such that if $\mathrm{diam}(X)\leq \pi -\varepsilon$ then $\lambda_1\geq N-1 +\delta_1$. 
\end{lem}

\begin{proof}
We prove this result by a contradiction argument. If the result does not hold, then for a fixed $\varepsilon$ there exists a sequence $\{(X_i,d_i,m_i)\}_{i=1}^{\infty}$ of $\rcd{N-2}{N-1}$ spaces (which we can assume to satisfy $m_i(X_i)=1$ without loss of generality) with uniformly bounded diameter by $\pi-\varepsilon$, such that $N-1\leq \lambda_1(X_i) \leq N-1 + 1/i$. 

By compactness, there exists an $\rcd{N-2}{N-1}$ space $(X_{\infty},d_{\infty},m_{\infty})$ such that, up to passing to a subsequence, $X_i$ converges to $X_{\infty}$ in the measured Gromov-Hausdorff sense. Therefore by the continuity of eigenvalues under measure Gromov-Hausdorff convergence in \cite{GigMonSav13}, it follows that  $\lambda_1(X_{\infty})=N-1$. However, $\mathrm{diam}(X)\leq \pi - \varepsilon$ which is a contradiction to the Obata Theorem in \cite{KetObata}. 
\end{proof} 

This result can be combined with Proposition \ref{pro-charact-harmonic-fns} to obtain the following decay property for harmonic functions on singular cones.

\begin{lem}
\label{lem-quant-average-energy-estimate}
Let $(X,d,m)$ be an $\rcd{N-2}{N-1}$ space such that $\mathrm{diam}(X)\leq \pi-\varepsilon$ for some $0<\varepsilon<\pi$. Let $(C,d_c,m_c):=\mathrm{Con}_0^N(X)$. Then,  there  exists $0<\delta_2=\delta_2(N,\varepsilon)<1$ such that if  $u:B_1(O)\subset C\to\mathbb{R}$ is a \(W^{1,2}\) harmonic function,  then for \(R\in (0,1]\),
\[
\fint_{B_{R}(O)}|\nabla u|^2\, \mathrm{d}m_c \leq R^{\delta_2} \fint_{B_1(O)}|\nabla u|^2\, \mathrm{d}m_c.
\]
\end{lem}

\begin{proof}
By \eqref{e:dec} we can write 
\[
u(r,x)=\sum_{i=0}^{\infty} a_i r^{\alpha_i} \varphi_i(x).
\]
Since the series is convergent in \(W^{1,2}\) and \(\{\varphi_i\}_i\) is an orthogonal system in \(L^2(X)\) and in  \(W^{1,2}(X)\), one gets  for \(R\in [0,1]\), 
\begin{equation}\label{e:francia}
\begin{split}
\int_{B_R(O)} |\nabla u|_c^2\, \mathrm{d}m_c &= \sum_{i=1}^{\infty}\sum_{j=1}^{\infty} a_ia_j\left( \int_{B_R(O)}r^{\alpha_i+\alpha_j}\inner{\varphi_i}{\varphi_j}_c + \varphi_i\varphi_j\inner{r^{\alpha_i}}{r^{\alpha_j}}_c	 \,\mathrm{d}m_c\right)
\\
&=\sum_{i=1}^{\infty} a_i^2\left( \int_{B_R(O)} r^{2\alpha_i}\abs{\nabla\varphi_i}^2_c + \varphi_i^2\abs{\nabla r^{\alpha_i}}^2_c \,\mathrm{d}m_c\right)
\\
&=\sum_{i=1}^{\infty} a_i^2\left(\int_0^R r^{2\alpha_i+N-3} \, \mathrm{d}r\int_X \abs{\nabla \varphi_i}_X^2\,\mathrm{d}m_X + \int_0^R\alpha_i^2 r^{2\alpha_i+N-3}\mathrm{d}r \int_X \varphi_i^2\,\mathrm{d}m \right)
\\ 
&= \sum_{i=1}^{\infty}\frac{a_i^2(\lambda_i + \alpha_i^2)R^{2\alpha_i+N-2}}{2\alpha_i + N -2} 
\end{split}
\end{equation}
where we have assumed without loss of generality that  \(\varphi_i\) are \(L^2\)-normalized, we have used  \eqref{e:measurewarping} and \eqref{e:gradientwarping} and the equality
\[
\lambda_i=\lambda_i\int_X \varphi^2_i  =-\int_X \varphi_i\Delta_X \varphi \mathrm{d}m_X=\int_X \abs{\nabla\varphi_i}^2_X\mathrm{d}m_X.
\]
Since  \(m_c(B_R(O))= \frac{R^{N}m(X)}{N}\), and the \(\alpha_i\) are increasing,  \eqref{e:francia} implies that 
\[
\fint_{B_R(O)} |\nabla u|_c^2\le R^{2\alpha_1-2} \fint_{B_1(O)} |\nabla u|_c^2.
\]
By  \cref{lem-quant-eigenvalue-bound}, \(\lambda_1\ge N-1+\delta_1\) and thus \(\alpha_1\ge 1+\delta_2\) for a suitable \(\delta_2>0\). This  concludes the proof.
\end{proof}
 
The next result shows that for tangent cones at a sharp point the assumptions of  \cref{lem-quant-eigenvalue-bound} are always satisfied.

\begin{lem}\label{lm:conebound} Let $(X,d,m)$ be an $\mathsf{RCD}(K,N)$ space and $x\in X$ with $\vartheta(x)<\infty$. Then the family of tangent cones $\mathrm{Con}_0^N(Y)$  is compact with respect to the $\mathrm{pmGH}$ topology if the vertex is chosen as the base point. In particular if  \(x\) is a sharp point, i.e. if every tangent cone has section \(Y\) with \(\diam (Y)<\pi\), then there exists \(\varepsilon=\varepsilon(N,K,x)\) such that  $\diam(Y)<\pi-\varepsilon$ for every tangent cone $\mathrm{Con}_0^N(Y)$ at $x$.
\end{lem}

\begin{proof}

Let $(C_i, d_i, m_i,O)$ be a sequence of tangent cones at $x$ converging in the pmGH-topology to a pointed space $(C,d,m,O)$. Consider for each $i$, a sequence of blowups 
\[
(X,d/r_{i,j},m/r_{i,j}^N,x)\quad\overset{\mathrm{pmGH}}{\longrightarrow}\quad (C_i,d_i,m_i, O).
\] 
 By a diagonalization argument there exists a sequence $r_{i_k,j_k}\to 0$ such that the corresponding sequence $(X,d/r_{i_k,j_k},m/r_{i_k,j_k}^N,x)$ converges to $C$ as $k\to \infty$. Whence, $C$ is a tangent space at $x$. 

To prove the second part of the statement, let us assume by contradiction that for a sequence $\varepsilon_i\to 0$, there exist tangent cones $\mathrm{Con}_0^N(Z_i)$ at $x$ such that $\mathrm{diam}(Z_i)\geq\pi-\varepsilon_i$. Therefore, by the compactness proved in the previous paragraph, there exists a subsequence $Z_{i_k}$ converging in the pmGH-topology to a tangent cone $\mathrm{Con}_0^N(Z)$ at $x$ for which  $\mathrm{diam}(Z)=\pi$, a contradiction. 
\end{proof}

The following is the key step in the proof  of  \cref{thm-mwug-vanishes}. In the following, given $x\in X$ and $R>0$, we use the notation $D(\Delta, B_R(x))$ for the family of $W^{1,2}(B_R(x))$-functions admitting a Laplacian on $B_R(x)$ in the sense of \cite[Definition 2.16]{AmbHon}.

\begin{pro}
\label{pro-energy-estimate-rcdkn}
Let $(X,d,m)$ be an $\rcd{K}{N}$ space and $x_0\in X$ with $\vartheta(x_0)<\infty$ be a sharp point.  
Then, there exists $\delta_0=\delta(K,N,x_0)>0$ and $R_0(x_0)>0$ such that if  $R\leq R_0$ and  $u\in D(\Delta, B_R(x_0))$  is such that 
\begin{equation}\label{e:1}
R^2\fint_{B_R(x_0)}\abs{\Delta u}^2 \, \mathrm{d}m \leq \delta_0 \fint_{B_R(x_0)} \abs{\nabla u}^2\, \mathrm{d}m,
\end{equation}
 then
\begin{equation}\label{e:2}
\fint_{B_{R/2}(x_0)} \abs{\nabla u}^2 \,\mathrm{d}m \leq (1-\delta_0)\fint_{B_R(x_0)}\abs{\nabla u}^2\,\mathrm{d}m.
\end{equation}
In particular for \(R\le R_0\) we have  
\begin{equation}\label{e:3}
\fint_{B_{R/2}(x_0)} \abs{\nabla u}^2 \,\mathrm{d}m \leq (1-\delta_0)\fint_{B_R(x_0)}\abs{\nabla u}^2\,\mathrm{d}m+C(N,K,R_0, \delta_0)  R^2\fint_{B_R(x_0)}\abs{\Delta u}^2 \, \mathrm{d}m
\end{equation}
\end{pro}

\begin{proof}
The proof is  by contradiction. Let us assume that there exists a sequence of radii $R_i\downarrow 0$ and functions $u_i\in D(\Delta, B_{R_i}(x_0))$ such that 
\begin{equation*}
R_i^2\fint_{B_{R_i}(x_0)}\abs{\Delta u_i}^2 \, \mathrm{d}m \leq \frac{1}{i} \fint_{B_{R_i}(x_0)} \abs{\nabla u_i}^2\, \mathrm{d}m,
\end{equation*}
and
\[
\fint_{B_{R_i/2}(x_0)} \abs{\nabla u_i}^2 \,\mathrm{d}m \geq \Big(1-\frac{1}{i}\Big)\fint_{B_{R_i}(x_0)}\abs{\nabla u_i}^2\,\mathrm{d}m.
\]
We consider the scaled pointed spaces by $(X_i,d_i,m_i,x_i):=(X,d/R_i, m/R_i^N, x_0)$ and, up to subsequence, we asume that they converge to   a tangent cone of the form $(\mathrm{Con}_0^N(Y),d_c,m_c)$. In particular by  \cref{lm:conebound}, \(\diam (Y)\le \pi-\varepsilon_0\) for a suitable \(\varepsilon_0=\varepsilon_0(K,N,x_0)>0\). We consider the family of functions
\[
v_i:=\frac{u_i - \fint_{B_1^i(x_0)} u_i \, \mathrm{d}m_i}{\left(\fint_{B^i_1(x_0)} \abs{\nabla u_i}^2\, \mathrm{d} m_i\right)^{1/2}}, 
\]
where \(B_1^i(x_0)\) is the unit ball of \((X,d/R_i)\). Observe that the result holds trivially if the denominator in the definition of $v_i$ vanishes. Note that, by trivial scaling,  we have 
\[
\|\Delta v_i\|_{L^2(B_1^i(x_0))}\to 0\,,\qquad \fint_{B^i_1(x_0)} \abs{\nabla v_i}^2\, \mathrm{d} m_i=1\,,\qquad  \fint_{B_{1/2}^{i}(x_0)}\abs{\nabla v_i}^2\ge 1-\frac{1}{i}.
\]
By the local Poincar\'e inequality, \cite{RajalaPoincare}, the \(v_i\) are bounded in \(L^2\) and hence by \cite[Theorem 4.2]{AmbHon} they converge, up to subsequences,  to some \(v\in W^{1,2}(B^C_1(O))\) where \(B^C_1(O)\) is the unit ball of \(\mathrm{Con}_0^N(Y)\). Therefore it follows, again by  \cite[Theorem 4.2]{AmbHon}, that
\[
\Delta v=0\qquad \text{and}\qquad\fint_{B^C_1(O)} \abs{\nabla v}_c^2\, \mathrm{d} m_c\le 1.
\]
Furthermore,  by \cite[Theorem 4.4]{AmbHon} the \(v_i\) converges ``strongly'' in \(W^{1,2}\), hence:
\[
1\leq\lim_i \fint_{B_{1/2}^{i}(x_0)}\abs{\nabla v_i}^2= \fint_{B_{1/2}^{C}(O)}\abs{\nabla v}_c^2\mathrm{d} m_c.
\]
However, by  \cref{lem-quant-average-energy-estimate} and since \(\diam (Y)\le \pi-\varepsilon_0\),
\[
1\leq \fint_{B_{1/2}^{C}(O)}\abs{\nabla v}_c^2\mathrm{d} m_c\le \Bigl(\frac{1}{2}\Bigr)^{\delta_2}\fint_{B_{1}^{C}(O)}\abs{\nabla v}_c^2\mathrm{d} m_c\le \Bigl(\frac{1}{2}\Bigr)^{\delta_2}<1,
\]
where \(\delta_2=\delta_2(K,N,\varepsilon_0)\). This contradiction concludes the proof of \eqref{e:2}. To prove \eqref{e:3} we note that by the first part, it trivially holds  if \eqref{e:1} is  satisfied. Otherwise,  if \eqref{e:1} is not  satisfied we have 
\[
\fint_{B_{R/2}(x_0)} \abs{\nabla u}^2\, \mathrm{d}m\le \frac{m(B_R(x_0))}{m(B_{R/2}(x_0))} \fint_{B_{R}(x_0)} \abs{\nabla u}^2\, \mathrm{d}m \le \frac{m(B_R(x_0))}{m(B_{R/2}(x_0))\delta_0}R^2\fint_{B_R(x_0)}\abs{\Delta u}^2 \, \mathrm{d}m. 
\]
Inequality \eqref{e:3} now follows by combining the above inequality with the fact that 
\[
\frac{m(B_R(x_0))}{m(B_{R/2}(x_0))}\le C(R_0,N,K)
\]
by the local doubling property of \(\rcd{K}{N}\) spaces, \cite{Vil}.
\end{proof}

We can now prove \cref{thm-mwug-vanishes}:

\begin{proof}[Proof of  \cref{thm-mwug-vanishes}]
The H\"older inequality implies that 
\[
\fint_{B_{R}(x_0)} \abs{\Delta  u}^2\mathrm{d}m\le m(B_R(x_0))^{-2/p}\Bigl(\int   \abs{\Delta u}^p\mathrm{d}m\Big)^{\frac{2}{p}}.
\]
By the Bishop-Gromov inequality, \(m(B_R(x_0)) \ge C(\bar R,N,K)R^N\), which combined with the above inequality yields
\[
R^{2} \fint_{B_{R}(x_0)} \abs{\Delta u}^2\mathrm{d}m\le C R^{\frac{2(p-N)}{p}},
\]
where \(C\) is a constant which depends on \(\bar R\) and on the \(L^p\)-norm of \(\Delta u\) in \(B_{\bar R}(x_0)\). Combining this last inequality with \eqref{e:3} we get that for all \(R\le \min\{R_0, \bar R\}\),
\[
\fint_{B_{R/2}(x_0)} \abs{\nabla u}^2\le (1-\delta_0)\fint_{B_{R}(x_0)} \abs{\nabla u}^2+ C R^{\frac{2(p-N)}{p}}.
\]
Since \(p>N\), the latter inequality immediately implies that
\[
\limsup_{R\to 0} \fint_{B_{R}(x_0)} \abs{\nabla u}^2=0.
\]
\end{proof}

\section{Proofs of  \cref{cor:noharm} and \cref{cor:nocz}}
To prove  \cref{cor:noharm}  we will argue by contradiction:  if the desired sequence of spaces and functions does not exist, one would then deduce an a priori bound for all manifolds with curvature bounded from below by $0$. This estimate will be stable by (measure) Gromov-Hausdorff convergence, yielding that the modulus of the gradient of harmonic functions is continuous on all spaces arising as a limit. On the other hand, by  \cref{thm-mwug-vanishes} it has to vanish at sharp singular points. Since one can easily construct a limit space with a dense set of sharp singular points, this gives the desired contradiction, by choosing  a sequence converging to a non-constant limit. 

The proof of \cref{cor:nocz} goes along  the same lines, by noting that the constant in the  Sobolev-Morrey embedding only depends on the upper bound of the dimension and on a lower bound on the Ricci curvature. Hence,  an a priori \(L^p\)-bound for \(p>N\)  implies a uniform bound on the \(C^{0,\frac{p-N}{N}}\) norm of \(|\nabla u|\).

In order to make the proofs clearer, let us recall here a few known facts:
\begin{enumerate}[(a)]
\item\label{i:badseq}  The boundary of an open convex set \(C\subset \mathbb{R}^{N+1}\) is an Alexandrov space with non-negative curvature when endowed with its intrinsic distance, (see \cite[Theorem 10.2.6]{BuragoBuragoIvanov}, noting that the proof applies for any dimension with minimal modifications). Furthermore there exists a sequence of smooth manifolds \(X_k\) with non-negative sectional curvature such that \(X_k\) converges to \(X\) in the Gromov-Hausdorff sense (see for example the proof of \cite[Theorem 1]{AKP}).

\item If \((X,d)\) is an Alexandrov space of dimension \(N\), then \((X,d, \mathcal H^N)\) is a non-collapsed  \(\rcd{0}{N}\) space, \cite{DePGig, Pet}.

\item\label{i:badset} For every \(N\) it is possible to construct the  boundary of a convex  set \(C\subset \mathbb R^N\), such that  \(X=\partial C\) which has a dense set of sharp singular points,  \cite[Example (2)]{OtsuShioya}.

\end{enumerate}

\begin{proof}[Proof of \cref{cor:noharm}] First note that if the conclusion of the theorem fails, arguing by contradiction this would mean that there exists a modulus of continuity \(\omega=\omega_{D,N}\) for all Riemannian manifolds \((X, g)\) with \(\sec(X) \ge 0\), dimension less than or equal to \(N\) and diameter bounded by \(D\) such that for all \(P\in M\) and all harmonic functions \(f:B_1^X(P)\to \mathbb R\)
\[
 \bigl|| \nabla f|(p)|- |\nabla f|(q)\bigr|\le \omega(d_{g}(p,q))\norm{f}_{W^{1,2}(B_1^X(P))}\qquad \text{for all \(p,q \in B_{1/2}^X(P)\).}
\]
By scaling and covering, this implies that for all \(s\in (0,1)\) there exists a modulus of continuity  \(\omega_s=\omega_{s,N,D}\) such that for all harmonic functions \(f:B_1^X(P)\to \mathbb R\) 
\begin{equation}\label{e:modulusofcontinuity}
 \bigl|| \nabla f|(p)|- |\nabla f|(q)\bigr|\le \omega_s(d_{g}(p,q))\norm{f}_{W^{1,2}(B_1^X(P))}\qquad \text{for all \(p,q \in B_{1-s}^X(P)\).}
\end{equation}
We now consider  a space \(X\) as in \ref{i:badset} above and a converging sequence as in \ref{i:badseq}. We also consider a point \(P\in X\) and a non constant  \(W^{1,2}\)-harmonic function defined on \(B_1^X(P)\). By \cite[Corollary 4.12]{AmbHon}, for all \(R\in (0,1)\) there are harmonic functions \(f_k: B_{R}^{X_k}(P_K)\to \mathbb{R}\) converging to \(f\) in the Mosco sense. By \eqref{e:modulusofcontinuity} one easily deduces that \(|\nabla f|\) is then a continuous function on the interior of \(B_1^X\). On the other hand, by \cref{thm-mwug-vanishes},  \(|\nabla f|\) has to vanish at all sharp  singular points. Since this set is dense,  this implies that \(|\nabla f|=0\) and thus that \(f\) is constant, a contradiction. 
\end{proof}

\begin{proof}[Proof of \cref{cor:nocz}]  We only sketch the proof, being essentially  the same of \cref{cor:noharm}. Indeed, by the Morrey embedding theorem (which holds uniformly on the class of \(\rcd{0}{N}\) spaces, \cite[Theorem 9.2.14]{HeiKosShaTys}) the validity of an a priori bound of the form 
\[
\norm{\nabla^2 f}_p\le C \norm{\Delta f}_p+\norm{\nabla f}_{p}
\]
with a constant which depends only on the lower bound on the curvature and the upper bound on the  dimension and the diameter, would yield a uniform estimate of the form 
\[
 \bigl|| \nabla f|(p)|- |\nabla f|(q)\bigr|\le C \norm{\Delta f}_p+\norm{\nabla f}_{p} d^{\frac{p-N}{p}}(p,q)\qquad \text{for all \(p,q \in X\),}
\]
where \(X\) is \emph{any} smooth manifold with \(\sec(X)\ge 0\) and  \(C\) depends only on an upper bound of the dimension and of the diameter. By considering a space \(X\) as in \ref{i:badset} and an approximating sequence \(X_k\) as in \ref{i:badseq} we obtain a contradiction as in in the proof of  \cref{cor:noharm}. 
\end{proof}

\bibliographystyle{amsplain}

\begin{thebibliography}{30}

\bibitem{AKP} S.~Alexander, V.~Kapovitch, and A.~Petrunin, 
\textit{An optimal lower curvature bound for convex hypersurfaces in Riemannian manifolds}, Illinois J. Math. 52 (2008), no. 3, 1031--1033. 

\bibitem{AmbrosioICM}  L.~Ambrosio, 
\textit{Calculus, heat flow and curvature-dimension bounds in metric measure spaces}. Proceedings of the ICM2018. Available at \texttt{http://cvgmt.sns.it/paper/3779/}.


\bibitem{AmbrosioGigliSavare11-2} L.~Ambrosio, N.~Gigli, and G.~Savar{\'e}, 
\textit{ Metric measure spaces with {R}iemannian {R}icci curvature bounded from below}, Duke Math. J., 163 (2014), pp.~1405--1490.

\bibitem{AmbHon} L.~Ambrosio and S.~Honda,
 \textit{Local spectral convergence in $\mathsf{RCD}^*(K,N)$ spaces}, preprint (2017) \texttt{arXiv:1703.04939 [math.MG]}



\bibitem{BuragoBuragoIvanov} D.~Burago, Y.~Dmitri, and S.~Ivanov, 
\textit{A course in metric geometry}, Graduate Studies in Mathematics, 33. American Mathematical Society, Providence, RI, 2001.

\bibitem{BuragoGromovPerelman}  Y.~Burago, M.~Gromov, and G.~ Perelman, 
\textit{A. D. Aleksandrov spaces with curvatures bounded below}, (Russian) ; translated from Uspekhi Mat. Nauk 47 (1992), no. 2(284), 3--51, 222 Russian Math. Surveys 47 (1992), no. 2, 1--58.





\bibitem{Cheeger-Colding96} J.~Cheeger and T.~H. Colding, 
\textit{Lower bounds on {R}icci curvature and the almost rigidity of warped products}, Ann. of Math. (2), 144 (1996), pp.~189--237.

\bibitem{Cheeger-Colding97I} J.~Cheeger and T.~H. Colding, 
\textit{On the structure of spaces with {R}icci curvature bounded below. {I}}, J. Differential Geom., 46 (1997), pp.~406--480.

\bibitem{Cheeger-Colding97II} J.~Cheeger and T.~H. Colding, 
\textit{On the structure of spaces with {R}icci curvature bounded below. {II}}, J. Differential Geom., 54 (2000), pp.~13--35.

\bibitem{Cheeger-Colding97III} J.~Cheeger and T.~H. Colding, \textit{On the structure of spaces with {R}icci curvature bounded below. {III}}, J. Differential Geom., 54 (2000), pp.~37--74.

\bibitem{Colding96} T.~H. Colding, 
\textit{Large manifolds with positive {R}icci curvature}, Invent. Math., 124 (1996), pp.~193--214.

\bibitem{Colding96b} T.~H. Colding, 
\textit{Shape of manifolds with positive {R}icci curvature}, Invent. Math., 124 (1996), pp.~175--191.

\bibitem{Colding97} T.~H. Colding, 
\textit{Ricci curvature and volume convergence}, Ann. of Math. (2), 145 (1997), pp.~477--501.

\bibitem{DePGig} G.~De~Philippis and N.~Gigli, 
\textit{From volume cone to metric cone in the nonsmooth setting}, Geom. Funct. Anal. 26 (2016), no. 6, 1526--1587.

\bibitem{DePGigNonCol} G.~De~Philippis and N.~Gigli,
\textit{Non-collapsed spaces with Ricci curvature bounded from below}, preprint (2017) \texttt{arXiv:1708.02060 [math.MG]}.

\bibitem{EKS} M.~Erbar, K.~Kuwada, K.~T.~Sturm, 
\textit{On the equivalence of the entropic curvature-dimension condition and Bochner's inequality on metric measure spaces}, Invent. math. (2015)  Vol. 201(3)  993--1071.


\bibitem{Gigli14} N.~Gigli,
\textit{Nonsmooth differential geometry - an approach tailored for spaces with {R}icci curvature bounded from below}, Mem. Amer. Math. Soc., 251 (2014), pp.~v+161.

\bibitem{GigDiffStruc} N.~Gigli, 
\textit{On the differential structure of metric measure spaces and applications}, Mem. Amer. Math. Soc. 236 (2015), no. 1113, vi+91 pp.

\bibitem{GigLecNotes} N.~Gigli, 
\textit{Lecture notes on differential calculus on $\mathsf{RCD}$ spaces}, preprint (2017) \texttt{arXiv:1703.06829 [math.DG]}. 

\bibitem{GigMonSav13} N.~Gigli, A.~Mondino, and G.~Savar\'e, 
\textit{Convergence of pointed non-compact metric measure spaces and stability of Ricci curvature bounds and heat flows}, Proc. Lond. Math. Soc. (3) 111 (2015), no. 5, 1071--1129. 


\bibitem{GigHan} N.~Gigli and B.~Han, 
\textit{Sobolev spaces on warped products}, preprint (2017), \texttt{arXiv:1512.03177 [math.FA]}

\bibitem{GunPig} B.~G\"uneysu and S.~Pigola, 
\textit{The Calder\'on-Zygmund inequality and Sobolev spaces on noncompact Riemannian manifolds}, Adv. Math. 281 (2015), 353--393.

\bibitem{HeiKosShaTys} J.~Heinonen, P.~Koskela, N.~Shanmugalingam and, J.~Tyson, \textit{Sobolev spaces on metric measure spaces}, volume 27 of New Mathematical Monographs, Cambridge University Press, Cambridge 2015.


\bibitem{JiangKoskelaYang} R.~Jiang, P.~Koskela and, D.~Yang,
\textit{Isoperimetric inequality via Lipschitz regularity of Cheeger-harmonic functions}, J. Math. Pures Appl. (9) 101 (2014), no. 5, 583--598.

\bibitem{Kell} M.~Kell, 
\textit{A Note on Lipschitz Continuity of Solutions of Poisson Equations in Metric Measure Spaces}, preprint (2013) \texttt{arXiv:1307.2224 [math.MG]}

\bibitem{Ket} C.~Ketterer, 
\textit{Cones over metric measure spaces and the maximal diameter theorem}, J. Math. Pures Appl. (9) 103 (2015), no. 5, 1228--1275. 

\bibitem{KetObata} C.~Ketterer, 
\textit{Obata's rigidity theorem for metric measure spaces}, Anal. Geom. Metr. Spaces 3 (2015), 278--295.


\bibitem{Huang}  X-T.~Huang, 
\textit{On the asymptotic behavior of the dimension of spaces of harmonic functions with polynomial growth} Journal für die reine und angewandte Mathematik (Crelles Journal), (2018) \texttt{doi:10.1515/crelle-2018-0029}

\bibitem{PLi} P.~Li, 
\textit{Geometric Analysis}, Cambridge Studies in Advanced Mathematics, 134. Cambridge University Press, Cambridge, (2012) 


\bibitem{Li} S.~Li, 
\textit{Counterexamples to the $L^p$-Calder\'on-Zygmund estimate on open manifolds}, preprint (2019), \texttt{arXiv:1902.10913 [math.AP]}

\bibitem{LotVil} J.~Lott and C.~Villani, 
\textit{Ricci curvature for metric-measure spaces via optimal transport}, Annals of Math. (2) 169 (2009), 903--991.

\bibitem{MitYam} A.~Mitsuishi and T.~Yamaguchi, 
\textit{Collapsing three-dimensional closed Alexandrov spaces with a lower curvature bound}, Trans. Amer. Math. Soc. 367 (2015), no. 4, 2339--2410

\bibitem{Ohta} S-i~Ohta, 
\textit{Finsler interpolation inequalities}, Calc. Var. Partial Differential Equations 36 (2009), no.2, 211--249

\bibitem{OtsuShioya} Y.~Otsu and T.~Shioya, 
\textit{The Riemannian structure of Alexandrov spaces}, J. Differential Geom. 39 (1994), no. 3, 629--658.
	
\bibitem{Pet} A.~Petrunin, 
\textit{Alexandrov meets Lott-Villani-Sturm}, M\"unster J. Math. 4 (2011), 53--64.	
	
\bibitem{RajalaPoincare} T.~Rajala, 
\textit{Local Poincar\'e inequalities from stable curvature conditions on metric spaces}, Calc. Var. PDE, 44 (2012), 477–494.	


\bibitem{Stu1} K.~T.~Sturm, 
\textit{On the geometry of metric measure spaces. I}, Acta Math. 196 (2006), 65--131.

\bibitem{Stu2} K.~T.~Sturm, 
\textit{On the geometry of metric measure spaces. II}, Acta Math. 196 (2006), no. 1, 133--177.

\bibitem{Vil} C.~Villani, 
\textit{Optimal transport. Old and new}, vol. 338 of Grundlehren der Mathematischen Wissenschaften, Springer-Verlag, Berlin, (2009).
 
\end{thebibliography}

\end{document}